\documentclass[11pt,a4paper,reqno]{amsart}

\usepackage{mathrsfs}
\usepackage{a4wide}
\usepackage{latexsym}
\usepackage{xcolor}
\usepackage[nobysame,initials]{amsrefs}
\usepackage{hyperref}

\usepackage{orcidlink}

\usepackage{times}

\allowdisplaybreaks

\author[K.J. B\"or\"oczky]{K\'aroly J. B\"or\"oczky}
\address{Alfr\'ed R\'enyi Institute of Mathematics, Re\'altanoda utca 13-15, 1053, Budapest, Hungary}  

\author[F. Fodor]{Ferenc Fodor$^{\orcidlink{0000-0001-9747-1981}}$} 
\address{Bolyai Institute, University of Szeged, Aradi v\'ertan\'uk tere 1, H-6720 Szeged, Hungary}
\email{fodorf@math.u-szeged.hu}

\author[P. Kalantzopoulos]{Pavlos Kalantzopoulos}
\address{University of Waterloo, Department of Pure Mathematics}
\email{pkalantzpoulos@uwaterloo.ca}

\title[Equality in Liakopoulos's generalized dual Loomis--Whitney inequality]{Equality in Liakopoulos's generalized dual Loomis--Whitney inequality via Barthe's Reverse Brascamp--Lieb inequality}

\newcommand{\R}{\mathbb{R}}

\newcommand{\dep}{\mathrm{dep}}
\newcommand{\dx}{{\rm d}}

\DeclareMathOperator{\conv}{conv}
\DeclareMathOperator{\relint}{relint}
\DeclareMathOperator{\lin}{lin}
\DeclareMathOperator{\inti}{int}
\DeclareMathOperator{\supp}{supp}
\DeclareMathOperator{\cl}{cl}

\newtheorem{lemma}{Lemma}
\newtheorem{theo}[lemma]{Theorem}

\theoremstyle{definition}

\theoremstyle{remark}
\newtheorem*{remarks}{Remarks}
\newtheorem{remark}[lemma]{Remark}

\begin{document}

\begin{abstract}
We use the characterization of the case of equality in  Barthe's Geometric Reverse Brascamp–Lieb inequality to characterize equality in Liakopoulos's volume estimate in terms of sections by certain lower-dimensional linear subspaces.
\end{abstract}

\maketitle

\section{Introduction and results
}
\label{secBT}

Geometric inequalities, in which the volume of a shape is given and we search for the extremal value of a quantity, have been considered since the time of the ancient Greeks and have been at the center of attention in Geometry and Analysis for the last two centuries.  The most classical one is probably the Isoperimetric inequality. Let $B^n\subset\R^n$ be the origin-centered Euclidean unit ball, and let $X\subset\R^n$ be a compact set such that ${\rm int}\,X\neq \emptyset$ and its boundary $\partial X$ is an $(n-1)$-dimensional Lipschitz manifold. Let $|X|$ denote the Lebesgue measure of $X$. We define the surface area $S(X)$ of $X$ as $S(X)=\lim_{r\to 0^+}\frac{|X+rB^n|-|X|}r$. The Isoperimetric inequality states that 
$$
S(X)\geq n|B^n|^{\frac1n}\cdot |X|^{\frac{n-1}n},
$$ 
with equality if and only if $X$ is a Euclidean ball.
One of the major results of Bent Fuglede was that he proved the optimal stability version of the Isoperimetric inequality for star-shaped $X\subset\R^n$ with Lipschitz boundary in terms of the Hausdorff distance in a series of papers \cites{Fug86,Fug89,Fug91} in 1986--1991
(see Figalli, Maggi, Pratelli
\cite{FMP10} and Fusco, Maggi, Pratelli \cite{FMP08} for optimal stability results in terms of volume difference).
Here, the Isoperimetric inequality is a consequence of the Brunn–Minkowski inequality
$$
|\alpha X+\beta Y|^{\frac1n}\geq \alpha |X|^{\frac1n}+\beta|Y|^{\frac1n}
$$
for compact sets $X,Y\subset\R^n$ with $|X|,|Y|>0$ and $\alpha,\beta>0$, where equality holds if and only if $X$ and $Y$ are convex and homothetic; namely, $Y=\lambda X+z$ for some $\lambda>0$ and $z\in\R^n$. See Figalli, Maggi, Pratelli
\cite{FMP10} for an elegant proof of the Brunn--Minkowski inequality including equality conditions and even a stability version based on optimal transport. In turn, the Brunn--Minkowski inequality has a functional form that is called the Pr\'ekopa--Leindler inequality, which we recall in a form most suitable for use in this paper: If $\sum_{i=1}^kc_i=1$ holds for $c_1,\ldots,c_k>0$ with $k\geq 2$, then for any non-negative $f_i\in L_1(\R^n)$, $i=1,\ldots,k$, we have
\begin{equation}
\label{Prekopa-Leindler}
\int_{\R^n}^*\sup_{x=\sum_{i=1}^kc_ix_i}\;\prod_{i=1}^kf_i(x_i)^{c_i}\,dx
\geq \prod_{i=1}^k\left(\int_{\R^n}f_i\right)^{c_i},
\end{equation}
where $\int_{\R^n}^*$ stands for the outer integral, as the integrand on the left hand side may not be measurable (see Gardner \cite{gardner} for a survey on various aspects of the Brunn–Minkowski inequality and the Pr\'ekopa–Leindler inequality). If $\int_{\R^n}f_i>0$ for all $i=1,,\ldots, k$, then equality  in the Pr\'ekopa--Leindler inequality \eqref{Prekopa-Leindler} implies that there exist $a_1,\ldots,a_k>0$,
$w_1,\ldots,w_k\in\R^n$ and
a log-concave function $\varphi$  such that $\sum_{i=1}^kc_i w_i=o$ and
$f_i(x)=a_i\varphi(x+w_i)$ for a.e. $x\in\R^n$ (see Dubuc \cite{Dub77} and Balogh, Krist\'aly \cite{BaK18} for the equality case, and Figalli,  van Hintum, and Tiba \cite{FHTb} for a stability version). 
Here, a function $f:\,\R^n\to [0,\infty)$ is {\it log-concave} if
$f((1-\lambda)x+\lambda y)\geq f(x)^{1-\lambda}f(y)^\lambda$ holds
 for $x,y\in\R^n$ and $\lambda\in(0,1)$; or equivalently, $f=e^{-\varphi}$ for a convex function 
$\varphi:\,\R^n\to(-\infty,\infty]$. 
If, in addition,  $0<\int_{\R^n}f<\infty$, then $\supp f=\cl \{f>0\}$ is convex,  $f$ is continuous on $\inti\supp f$,  $f$ is bounded, and for any $t$ with $0<t<\sup f$, the open set $\{f>t\}$ is convex and bounded.

In this note, we characterize equality in Liakopoulos's inequality (Theorem~\ref{Liakopoulos-RBL}) as a first step towards a possible stability version.
However, the story of this paper  originates from the classical Loomis--Whitney inequality \eqref{Loomis-Whitney-ineq} proved in
\cite{LoW49}.

Let $e_1,\ldots,e_n$ denote an orthonormal basis of $\R^n$.
 For a compact set $K\subset\R^n$ with
${\rm dim}\,{\rm aff}\,K=m$, we write $|K|$ for the $m$-dimensional Lebesgue measure of $K$. For a proper linear subspace $E$ of $\R^n$ ($E\neq \R^n$ and $E\neq\{0\}$),
let $P_E$ denote the orthogonal projection onto $E$.

\begin{theo}[Loomis, Whitney]
\label{Loomis-Whitney}
If $K\subset \R^n$ is compact, ${\rm aff}\,K=\R^n$ and $|K|>0$, then
\begin{equation}
\label{Loomis-Whitney-ineq}
|K|^{n-1}\leq \prod_{i=1}^n|P_{e_i^\bot}K|,
\end{equation}
with equality if and only if
$K=\oplus_{i=1}^nK_i$, where $\mathrm{aff}\, K_i$ is a line parallel to $e_i$.
\end{theo}

In order to consider a generalization of the Loomis--Whitney inequality,
we set
$[n]:=\{1,\ldots,n\}$, and for a non-empty proper subset $\sigma\subset[n]$, we define
$E_\sigma=\lin \{e_i\}_{i\in\sigma}$. For $s\geq 1$, we say that the not necessarily distinct proper
non-empty subsets $\sigma_1,\ldots,\sigma_k\subset[n]$ form an $s$-uniform cover of $[n]$ if each
$j\in[n]$ is contained in exactly $s$ of $\sigma_1,\ldots,\sigma_k$. The Bollob\'as--Thomason inequality \cite{BoT95} is as follows.

\begin{theo}[Bollob\'as, Thomason]
\label{Bollobas-Thomason}
If $K\subset \R^n$ is compact, ${\rm aff}\,K=\R^n$ and
$\sigma_1,\ldots,\sigma_k\subset[n]$ form an $s$-uniform cover of $[n]$ for an integer $s\geq 1$, then
\begin{equation}
\label{Bollobas-Thomasson-ineq}
|K|^s\leq \prod_{i=1}^k|P_{E_{\sigma_i}}K|.
\end{equation}
\end{theo}

\begin{remarks}
\hfill
\begin{itemize}

\item[(i)] If $k=n$ and $s=n-1$, in which case we may assume that $\sigma_i=[n]\backslash \{i\}$, inequality \eqref{Bollobas-Thomasson-ineq} reduces to the
Loomis--Whitney inequality \eqref{Loomis-Whitney-ineq}.

\item[(ii)]
Bennett, Carbery, Christ and Tao introduced the notion of a ''Geometric Brascamp--Lieb datum" in \cite{BCCT08}. A collection of linear subspaces $E_1,\ldots,E_k\subset \R^n$, each of dimension at least one, together with positive constants $c_1,\ldots,c_k>0$ forms a Geometric Brascamp--Lieb datum if
\begin{equation}
\label{BL-data-def}
\sum_{i=1}^kc_iP_{E_i}=I_n,
\end{equation}
where $I_n$ is the identity map on $\R^n$. The term \emph{geometric} refers to the fact that the surjective linear maps $P_{E_i}\colon \R^n\to E_i$, $i=1,\ldots, k$ are orthogonal projections.

Now, the subspaces $E_{\sigma_i}:={\rm lin}\{e_j:j\in\sigma_i\}$ in Theorem~\ref{Bollobas-Thomason} satisfy
\begin{equation}
\label{GBLdata-BT}
\displaystyle \sum_{i=1}^k \frac1s\,P_{E_{\sigma_i}}=I_n;
\end{equation}
that is, they form a Geometric Brascamp--Lieb datum, where each coefficient is $\frac1s$. 
\end{itemize}
\end{remarks}

The equality case of
the  Bollob\'as--Thomason inequality (Theorem~\ref{Bollobas-Thomason}), based on Valdimarsson's characterization theorem \cite{Val08}, has been known by experts for some time without a published proof.
Let $s\geq 1$, and let $\sigma_1,\ldots,\sigma_k\subset[n]$ be an $s$-uniform cover
   of $[n]$. 
We say that the disjoint union
 $[n]=\tilde{\sigma}_1\sqcup\ldots\sqcup\tilde{\sigma}_\ell$
forms the partition of $[n]$ induced by
the $s$-uniform cover $\sigma_1,\ldots,\sigma_k$ if
$\{\tilde{\sigma}_1,\ldots,\tilde{\sigma}_\ell\}$ consists of all non-empty distinct subsets of $[n]$
of the form $\cap_{i=1}^k\sigma_{i,0}$, where for each $i$, either $\sigma_{i,0}=\sigma_i$ or  $\sigma_{i,0}=[n]\setminus\sigma_i$. 

\begin{theo}[Folklore]
\label{Bollobas-Thomason-eq}
Let $K\subset \R^n$ be compact with 
${\rm aff}\,K=\R^n$ and $|K|>0$, and let
$\sigma_1,\ldots,\sigma_k\subset[n]$ form an $s$-uniform cover of $[n]$ for an integer $s\in\{1,\ldots,n-1\}$.
Then equality holds in \eqref{Bollobas-Thomasson-ineq} if and only if
$K=\oplus_{i=1}^\ell P_{E_{\tilde{\sigma}_i}}K$,
where $\tilde{\sigma}_1,\ldots,\tilde{\sigma}_\ell$ form the partition of $[n]$ induced by
 $\sigma_1,\ldots,\sigma_k$.
\end{theo}

We note that  Ellis, Friedgut, Kindler, and Yehudayoff \cite{EFKY16} even proved a stability version of the Bollob\'as--Thomason inequality Theorem~\ref{Bollobas-Thomason} in the case relevant to information theory, and hence that of the Loomis--Whitney inequality as well. 

The real starting point of our note is Meyer's  dual Loomis--Whitney inequality in \cite{Mey88}, where equality holds for affine cross-polytopes.

\begin{theo}[Meyer]
\label{dual-Loomis-Whitney}
If $K\subset \R^n$ is a compact convex set with $o\in\mathrm{int}\, K$, then
\begin{equation}
\label{dual-Loomis-Whitney-ineq}
|K|^{n-1}\geq \frac{n!}{n^n}\prod_{i=1}^n\left |K\cap e_i^\bot\right |,
\end{equation}
with equality if and only if
$K=\conv \{\pm\lambda_ie_i\}_{i=1}^n$ for $\lambda_i>0$, $i=1,\ldots,n$.
\end{theo}

Let us observe that, unlike in the case of the Loomis-Whitney inequality, the assumption that $K$ is convex is necessary in Meyer's inequality \eqref{dual-Loomis-Whitney-ineq}. 

We note that various reverse and dual Loomis--Whitney type inequalities were proved by
Campi, Gardner, Gronchi \cite{CGG16}, Brazitikos \emph{et al.} \cites{BDG17,BGL18},
Alonso-Guti\'errez \emph{et al.} \cites{ABBC,AB}. In particular, according to \eqref{GBLdata-BT},
the following theorem, due to Liakopoulos \cite{Lia19}, is a generalization of Meyer's inequality on one hand,
and is the generalized  dual form of the Bollob\'as--Thomason inequality (Theorem~\ref{Bollobas-Thomason}) on the other hand.

\begin{theo}[Liakopoulos]
\label{Liakopoulos-RBL}
If $K\subset \R^n$ is a compact convex set with $o\in\inti K$, and
 the linear  subspaces $E_1,\ldots,E_k$ of $\R^n$ and real numbers $c_1,\ldots,c_k>0$,  with
${\rm dim}\,E_i=n_i\in\{1,\ldots,n-1\}$, form a Geometric Brascamp–Lieb datum; that is,
\begin{equation}
\label{highdimcond-Liakopoulos}
\sum_{i=1}^kc_iP_{E_i}=I_n,
\end{equation}
then
\begin{equation}
\label{Liakopoulos-RBL-ineq}
|K|\geq \frac{\prod_{i=1}^k(n_i!)^{c_i}}{n!}\cdot \prod_{i=1}^k\left |K\cap E_i\right |^{c_i}.
\end{equation}
\end{theo}

Like in Meyer's inequality, the convexity of $K$ is a necessary condition in Liakopoulos' inequality.

Given the  linear  subspaces $E_1,\ldots,E_k$ of $\R^n$ as part of a Geometric Brascamp--Lieb datum in Theorem~\ref{Liakopoulos-RBL}, we say that
$F=\cap_{i=1}^kE_{i,0}$ is an independent subspace if either
$E_{i,0}=E_i$ or $E_{i,0}=E_i^\bot$
holds for any $i=1,\ldots,k$, and ${\rm dim}\,F\geq 1$ (see also the description above Theorem~\ref{BLtheoequa}). In general, there may not exist any independent subspace associated with $E_1,\ldots,E_k$; however, if $F\neq F'$ are independent subspaces, then $F'\subset F^\bot$.

Our main result characterizes equality in Liakopoulos's Theorem~\ref{Liakopoulos-RBL}.

\begin{theo}
\label{Liakopoulos-RBL-equa}
We have equality in \eqref{Liakopoulos-RBL-ineq} in Theorem~\ref{Liakopoulos-RBL} if and only if
\begin{itemize}
  \item[(i)] $\oplus_{j=1}^\ell F_j=\R^n$ holds  for the independent subspaces $F_1,\ldots,F_\ell\subset\R^n$ associated with $E_1,\ldots,E_k$, and
\item[(ii)] $K=\conv\{K\cap F_{j}\}_{j=1}^\ell$.
\end{itemize}
\end{theo}

The structure of the paper is as follows. We introduce the Brascamp--Lieb inequality and Barthe's reverse form, which play a crucial role in the proof of Theorem~\ref{Liakopoulos-RBL-equa} in Section~\ref{secBLandRBL}. We prove Theorem~\ref{Liakopoulos-RBL-equa} in Section~\ref{secLiakopoulos-RBL-equa}.

\section{Equality in Barthe's Geometric Reverse Brascamp--Lieb inequality and some related results}
\label{secBLandRBL}

Let the linear subspaces $E_1,\ldots,E_k\subset\R^n$ and real numbers  $c_1,\ldots,c_k>0$  form a Geometric Brascamp--Lieb datum. If $\dim E_i=n_i$ for $i=1,\ldots, n$, then 
\begin{equation}
\label{sum-of-ci}
\sum_{i=1}^kc_in_i=n
\end{equation}
by comparing traces in \eqref{BL-data-def}.

The Geometric Brascamp--Lieb inequality (Theorem~\ref{BLtheo}) traces back to the works of Brascamp and Lieb \cite{BrL76} and Ball \cites{Bal89,Bal91} in the rank-one case, where $\dim E_i=1$ for every $i=1,\ldots,k$, and to Lieb \cite{Lie90}, Barthe \cite{Bar98} and Carbery, Christ, and Tao \cite{BCCT08} in the general setting. We note that a rank one Geometric Brascamp–Lieb datum
 is called a Parseval frame in computer science (cf. Casazza, Tran, and Tremain \cite{CTT20}).

\begin{theo}[Brascamp, Lieb, Ball, Barthe]
\label{BLtheo}
For any non-trivial linear subspaces  $E_1,\ldots,E_k$ of $\R^n$ and constants $c_1,\ldots,c_k>0$ satisfying
\eqref{BL-data-def}, and for non-negative $f_i\in L_1(E_i)$, we have
\begin{equation}
\label{BL}
\int_{\R^n}\prod_{i=1}^kf_i(P_{E_i}x)^{c_i}\,dx
\leq \prod_{i=1}^k\left(\int_{E_i}f_i\right)^{c_i}
\end{equation}
\end{theo}

\begin{remark}
When $E_1=\cdots=E_k=\R^n$, then the statement \eqref{BL} becomes H\"older's inequality, and $\sum_{i=1}^k c_i=1$ holds in that case (cf. \eqref{sum-of-ci}).
\end{remark}

According to \eqref{GBLdata-BT}, the Geometric Brascamp--Lieb inequality (Theorem~\ref{BLtheo}) directly yields the Bollob\'as--Thomason inequality (Theorem~\ref{GBLdata-BT}) by choosing $E_i=E_{\sigma_i}$, $c_i=\frac1s$ and $f_i=\mathbf{1}_{P_{E_{\sigma_i}}}$ for $i=1,\ldots,k$.
Equality in Theorem~\ref{BLtheo} occurs, for example, when each $f_i$ is a centered Gaussian density like  $f_i(x)=e^{-\pi|x|^2}$ for all $i=1,\ldots, k$, where $|\cdot|$ denotes the Euclidean norm on $\R^n$. 

Partial progress on the equality case of the Geometric Brascamp–Lieb inequality was made by Barthe \cite{Bar98}, Carlen, Lieb, and Loss \cite{CLL04}, and Bennett, Carbery, Christ, and Tao \cite{BCCT08}. The full characterization was later achieved by Valdimarsson \cite{Val08}. Before stating Valdimarsson's result, we establish the necessary notation. Assume that the subspaces $E_1,\ldots,E_k$ of $\R^n$  and constants $c_1,\ldots,c_k>0$ form a Geometric Brascamp--Lieb datum. We call a nontrivial linear subspace $V$  critical if 
\[ \sum_{i=1}^k c_i \dim(E_i \cap V) = \dim V.\]
According to \cite{BCCT08}, this criterion is equivalent to the decomposition
\[E_i = (E_i \cap V) + (E_i \cap V^\perp) \quad \text{for each } i=1,\ldots,k. \]

In order to characterize equality in Theorem~\ref{Liakopoulos-RBL}, we also recall the notions of independent and dependent subspaces associated with a Geometric Brascamp--Lieb datum. These notions were originally defined by Valdimarsson \cite{Val08}. Let $J$ denote the set of all functions from $\{1,\ldots,k\}$ to $\{0,1\}$. For $\varepsilon\in J$, let $F_{(\varepsilon)}=\cap_{i=1}^kE_i^{(\varepsilon(i))}$, where $E_i^{(0)}=E_i$ and $E_i^{(1)}=E_i^\bot$
for $i=1,\ldots,k$. If $\dim F_{(\varepsilon)}\geq 1$, then we call $F_{(\varepsilon)}$ independent.  
Note that if $F$ and $F'$ are two different independent subspaces, then $F'\subset F^\perp$. 
Let $J_0=\{j\in J\colon \dim F_{(\varepsilon)}\geq 1\}$.
We call the orthogonal complement $F_{\mathrm{dep}}$ of
$\oplus_{\varepsilon \in J_0}F_{(\varepsilon)}$ in $\R^n$ the dependent subspace.
In this way, $\R^n$ decomposes into a direct sum of pairwise orthogonal linear subspaces
\begin{equation}
\label{independent-dependent0}
\R^n=\left(\oplus_{\varepsilon \in J_0}F_{(\varepsilon)}\right)\oplus F_{\dep}.
\end{equation}
Observe that there may not exist any independent subspaces, in which case $\R^n=F_{\dep}$. It may also happen that the dependent subspace is trivial; then $\R^n=\oplus_{\varepsilon \in J_0}F_{(\varepsilon)}$. In general, all independent subspaces and the dependent subspace (if  $F_{\dep}$ is non-trivial) are critical subspaces.

For a non-zero linear subspace $L \subset \R^n$, we say that the matrix representing a linear transformation $A:L\to L$ is \emph{positive definite} if $A$ is symmetric (i.e., $A^T = A$) and $x^T A x > 0$ for all  $x \in L \setminus \{0\}$.

\begin{theo}[Valdimarsson]
\label{BLtheoequa}
For  any proper linear subspaces  $E_1,\ldots,E_k$ of $\R^n$ and real numbers $c_1,\ldots,c_k>0$ satisfying
\eqref{BL-data-def}, let us assume that equality holds in the Brascamp--Lieb inequality \eqref{BL}
for non-negative $f_i\in L_1(E_i)$, $i=1,\ldots,k$.
If $F_{\dep}\neq\R^n$, then let $F_1,\ldots,F_\ell$ be the independent subspaces, and if
$F_{\dep}=\R^n$, then let $\ell=1$ and $F_1=\{0\}$. Then there exist
$b\in F_{\dep}$ and $\theta_i>0$ for $i=1,\ldots,k$,
integrable non-negative $h_{j}:\,F_j\to[0,\infty)$  for $j=1,\ldots,\ell$, and a positive definite matrix
$A:F_{\dep}\to F_{\dep}$ such that
the eigenspaces of $A$ are critical subspaces and
\begin{equation}
\label{BLtheoequaform}
f_i(x)=\theta_i e^{-\langle AP_{F_{\dep}}x,P_{F_{\dep}}x-b\rangle}\prod_{F_j\subset E_i}h_{j}(P_{F_j}(x))
\mbox{ \ \ \  for Lebesgue a.e. $x\in E_i$}.
\end{equation}
On the other hand, if for any $i=1,\ldots,k$, $f_i$ is of the form as in \eqref{BLtheoequaform}, then equality holds in \eqref{BL} for $f_1,\ldots,f_k$.
\end{theo}

The subspaces $F_i$, $i=1,\ldots, \ell$ are called \emph{independent} to express the fact that the functions $h_j$ defined on $F_j$ may be selected arbitrarily and independently of one another in Theorem~\ref{BLtheoequa}.

Barthe \cite{Bar98} established a reverse version of the Geometric Brascamp–Lieb inequality. 

\begin{theo}[Barthe]
\label{RBLtheo}
For any non-trivial linear subspaces  $E_1,\ldots,E_k$ of $\R^n$ and $c_1,\ldots,c_k>0$ satisfying
\eqref{BL-data-def}, and for non-negative $f_i\in L_1(E_i)$, we have
\begin{equation}
\label{RBL}
\int_{\R^n}^*\sup_{x=\sum_{i=1}^kc_ix_i,\, x_i\in E_i}\;\prod_{i=1}^kf_i(x_i)^{c_i}\,dx
\geq \prod_{i=1}^k\left(\int_{E_i}f_i\right)^{c_i}.
\end{equation}
\end{theo}

\begin{remark}
If $E_1=\cdots=E_k=\R^n$, then Theorem~\ref{RBLtheo} becomes the Pr\'ekopa--Leindler inequality \eqref{Prekopa-Leindler}, because $P_{E_i}=I_n$ and thus $\sum_{i=1}^k c_i=1$ (cf. \eqref{sum-of-ci}).
\end{remark}

 We note that Barthe’s Geometric Reverse Brascamp--Lieb inequality (Theorem~\ref{RBLtheo}) was used by Liakopoulos to prove the dual form of Bollob\'as--Thomason inequality (Theorem~\ref{Liakopoulos-RBL}) by choosing $f_i = 1_{K \cap E_i}$ for $i = 1,\dots,k$. In Section~\ref{secLiakopoulos-RBL-equa}, we follow this argument to exploit the equality case.
 B\"or\"oczky, Kalantzopoulos, Xi \cite{BKX23} proved the following characterization of equality in Barthe's inequality \eqref{RBL}.

\begin{theo}[\cite{BKX23}]
\label{RBLtheoequa}
For linear subspaces  $E_1,\ldots,E_k$ of $\R^n$ and $c_1,\ldots,c_k>0$ satisfying
\eqref{BL-data-def},
if $F_{\dep}\neq\R^n$, then let $F_1,\ldots,F_\ell$ be the independent subspaces, and if
$F_{\dep}=\R^n$, then let $\ell=1$ and $F_1=\{0\}$.

If equality holds in Barthe's Geometric Reverse Brascamp--Lieb inequality \eqref{RBL}
for non-negative $f_i\in L_1(E_i)$ with $\int_{E_i}f_i>0$, $i=1,\ldots,k$, then
\begin{equation}
\label{RBLtheoequaform}
f_i(x)=\theta_i e^{-\langle AP_{F_{\dep}}x,P_{F_{\dep}}x-b_i\rangle}\prod_{F_j\subset E_i}h_{j}\left(P_{F_j}(x-w_i)\right)
\end{equation}
for Lebesgue a.e. $x\in E_i$, where
\begin{itemize}
\item $\theta_i>0$,
$b_i\in E_i\cap F_{\rm dep}$ and 
$w_i\in E_i$ for $i=1,\ldots,k$,
\item $h_{j}\in L_1(F_j)$ is non-negative for $j=1,\ldots,\ell$, and, in addition,
$h_j$ is log-concave if there exist $\alpha\neq \beta$ with $F_j\subset E_\alpha\cap E_\beta$,
\item $A:F_{\dep}\to F_{\dep}$ is a positive definite matrix such that
the eigenspaces of $A$ are critical subspaces.
\end{itemize}
On the other hand, if for any $i=1,\ldots,k$, $f_i$ is of the form as in \eqref{RBLtheoequaform} and equality holds for  all
$x\in E_i$ in \eqref{RBLtheoequaform}, then equality holds in \eqref{RBL} for $f_1,\ldots,f_k$.
\end{theo}

We discuss some special cases of Theorem~\ref{RBLtheoequa} that will be useful in the proof of Theorem~\ref{Liakopoulos-RBL-equa}. In Lemma~\ref{RBLtheoequa-remarks} we use the notation of Theorem~\ref{RBLtheoequa}.

\begin{lemma} 
\label{RBLtheoequa-remarks}
Let $f_1,\ldots,f_k$ form a set of extremizers of Barthe's Geometric Reverse Brascamp--Lieb inequality.
\begin{itemize}
\item[(i)] If for each $\alpha=1,\ldots,k$, the remaining subspaces $\{E_i\}_{i\neq \alpha}$ span $\R^n$,
then each $f_i$ is $\mathcal{H}^{n_i}$ a.e. equal to some log-concave function on $E_i$.

\item[(ii)] If each $f_i$ is log-concave, then one may assume that each $h_j$, $j=1,\ldots,\ell$, is log-concave, and hence \eqref{RBLtheoequaform} holds for all $x\in \mathrm{relint}\,\mathrm{supp}\, f_i$.

\item[(iii)] If each $f_i$ is even, then one may assume that each $b_i=w_i=o$ for $i=1,\ldots,k$, and each $h_j$ is even for $j=1,\ldots,\ell$ (and, in fact, each $h_j$ being even and log-concave can be assumed if each $f_i$ is even and log-concave).

\end{itemize}
\end{lemma}
\begin{proof} For (i), the condition that for any $\alpha=1,\ldots,k$, $\{E_i\}_{i\neq \alpha}$ spans $\R^n$ in Theorem~\ref{RBLtheoequa} is equivalent to saying that for any $F_j$, $j=1,\ldots,\ell$, there exist $\alpha\neq \beta$ with $F_j\subset E_\alpha\cap E_\beta$. 

In the argument for (ii) and (iii), if $F_{\dep}\neq \{o\}$, then we write $h_{0,i}(y)=e^{-\langle Ay,y-b_i\rangle}$ for $i=1,\ldots,k$ and $y\in E_i\cap F_{\dep}$, and if $F_j\subset E_i$ is an independent subspace, then we write $h_{j,i}\left(y\right)=h_{j}\left(y-P_{F_j}w_i\right)$ for $y\in F_j$.

For (ii), first we show that even if  some independent subspace $F_m$ is contained in a single $E_\alpha$, then $h_m$ is Lebesgue a.e. equal to some log-concave function on $F_m$. In this case, let $z\in E_\alpha\cap F_m^\bot$ such that $h_{j,\alpha}(P_{F_j}z)>0$ whenever $F_j\subset E_\alpha$ and $j\neq m$, and 
for Lebesgue  a.e. $x\in z+F_m$, we have
$$
f_\alpha(x)= 
\theta_\alpha h_{0,\alpha}\left(P_{F_{\rm dep}}z\right)
h_{m,\alpha}\left(P_{F_m} x \right)
\prod_{F_j\subset E_\alpha,\&\,j\neq m}h_{j,\alpha}\left(P_{F_j} z \right).
$$
As $f_\alpha$ is log-concave,  $h_{m,\alpha}$ agrees with a log-concave function for Lebesgue a.e. $y\in F_m$, and hence the same holds for $h_\alpha$. 

Now, as log-concave functions are continuous on their support, we may assume that each $h_j$ is actually log-concave, and hence 
\eqref{RBLtheoequaform} holds for all $x\in \relint(\supp f_i)$ for all $i=1,\ldots,k$.

To show (iii),  one replaces each $h_{j,i}$ with the function $y\mapsto \sqrt{h_{j,i}(y)\cdot h_{j,i}(-y)}$ that is log-concave if $h_{j,i}$ is log-concave. 
\end{proof}

Part (i) of Lemma~\ref{RBLtheoequa-remarks} states that if for every $E_\alpha$, the remaining subspaces $E_i$, $i\neq\alpha$ of the Brascamp--Lieb datum span $\R^n$, then the functions $h_j$ involved in the extremizers in Theorem~\ref{RBLtheoequa} are necessarily log-concave.
The reason behind this, as already explained in \cite{BKX23}, is as follows. If $F_j$ is one of the independent subspaces, then the coefficients satisfy $\sum_{E_i\supset F_j}c_i=1$. When the functions $f_1,\ldots,f_k$ have the form given in \eqref{RBLtheoequaform}, then the equality case of Barthe’s inequality \eqref{RBL} takes the form
\[
\int^*_{F_j}\sup_{x=\sum_{E_i\supset F_j}c_i x_i\atop x_i\in F_j}h_{j}\Big(x_i-P_{F_j}w_i\Big)^{c_i}\,dx=
\prod_{E_i\supset F_j}\left(\int_{F_j}h_{j}\Big(x-P_{F_j}w_i\Big)\,dx\right)^{c_i}
\left(= \int_{F_j} h_j(x)\,dx\right).
\]
Consequently, if there exist two distinct indices $\alpha\neq \beta$ such that $F_j\subset E_\alpha\cap E_\beta$, then the equality conditions in the Prékopa--Leindler inequality \eqref{Prekopa-Leindler} force each $h_j$ to be log-concave. In contrast, if there is some $\alpha\in{1,\ldots,k}$ for which $F_j\subset E_\beta^\perp$ for all $\beta\neq \alpha$, then no restriction is imposed on $h_j$, and in this case $c_\alpha=1$.

Finally, we note that the 
Brascamp–Lieb inequality and Barthe's reverse form have more general versions proved, for example, via optimal transport by Barthe \cite{Bar98} or using the heat equation by Barthe, Cordero-Erausquin \cite{BaC04} (see, for example, \cite{Bor25+} or \cite{BKX23} for a discussion what is known about the equality case), but our paper is only concerned with the Geometric versions.
The Brascamp--Lieb inequality has been applied in various fields in mathematics, like probability, harmonic analysis  and convex geometry, and
 even in number theory, for example, by Guo, Zhang \cite{GuZ19}. In addition, several versions of the Brascamp--Lieb inequality and Barthe's reverse form have been proved by
Balogh, Krist\'aly \cite{BaK18}
Barthe \cite{Bar04}, Barthe, Cordero-Erausquin \cite{BaC04},
Barthe, Cordero-Erausquin,  Ledoux, Maurey \cite{BCLM11},
Barthe, Wolff \cites{BaW14,BaW22},
 Bennett, Bez, Buschenhenke, Cowling, Flock \cite{BBBCF20},
Bobkov, Colesanti, Fragal\`a \cite{BCF14},
Bueno, Pivovarov \cite{BuP21},
Chen,  Dafnis, Paouris \cite{CDP15},  Courtade, Liu \cite{CoL21},
Duncan \cite{Dun21}, Ghilli, Salani \cite{GhS17},
Kolesnikov, Milman \cite{KoM}, Lehec \cite{Leh14},  Livshyts \cites{Liv21,Liv},
Lutwak, Yang, Zhang \cites{LYZ04,LYZ07},
Maldague \cite{Mal},  Marsiglietti \cite{Mar17}, Rossi, Salani \cites{RoS17,RoS19}.

\section{Proof of Theorem~\ref{Liakopoulos-RBL-equa}}
\label{secLiakopoulos-RBL-equa}

In order to prove Theorem~\ref{Liakopoulos-RBL-equa}, we recall the proof of Theorem~\ref{Liakopoulos-RBL}; see  Liakopoulos \cite{Lia19}*{Section~3}. We will use the following three simple observations. For a convex body $M\subset\R^n$ with $o\in\inti M$, the value of the gauge function $\|x\|_M$ on a vector $x\in\R^n$ is defined as $\|x\|_M=\inf \{\lambda\geq 0\colon x\in \lambda M\}$. Then, since $|M|=\frac 1n\int_{S^{n-1}}\|u\|^{-n}_M\, \dx u$, it holds that
\begin{equation}
\label{VKexponential}
\int_{\R^n}e^{-\|x\|_M}\,\dx x=\int_0	^\infty e^{-r}nr^{n-1}|M|\,\dx r=n! |M|,
\end{equation}
see Schneider \cite{Sch14}*{(1.54) on p. 58}.

We often use the fact that for a convex body $M\subset\R^n$ with $o\in\inti M$ and any linear subspace $F\subset\R^n$, we have $\|x\|_M=\|x\|_{M\cap F}$ for $x\in F$. This yields that if $\oplus_{j=1}^m G_j=\R^n$ holds  for the pairwise orthogonal linear subspaces $G_1,\ldots,G_m\subset\R^n$,
and  $M_j\subset G_j$ is a $\dim G_j$-dimensional compact convex set with $o\in\relint M_j$ for $j=1,\ldots,m$, then
$M=\conv \{M_1,\ldots,M_m\}$ satisfies
\begin{equation}
\label{sumMjnorm}
\|x\|_M=\sum_{i=1}^m\|P_{G_j}x\|_{M_j}
\end{equation}
for any $x\in\R^n$.

If there exist any independent subspaces associated with $E_1,\ldots,E_k$, then let these be
  $F_1,\ldots,F_\ell\subset\R^n$. 
  If there exists no independent subspace, then let $\ell=1$ and $F_1=\{o\}$. Next, let $\widetilde{F}=\oplus_{j=1}^\ell F_j$, and hence $F_0=\widetilde{F}^\bot$ is the dependent subspace associated with $E_1,\ldots, E_k$. 

We define
\begin{align}\label{defoff}
f(x)=e^{-\|x\|_K},
\end{align}
which is a log-concave function with $f(0)=1$, and it satisfies (\emph{cf.} \eqref{VKexponential}),
\begin{equation}
\label{ftonint}
\int_{\R^n}f(y)^n\,\dx y=\int_{\R^n}e^{-n\|y\|_K}\,\dx y=\int_{\R^n}e^{-\|y\|_{\frac{1}{n}K}}\,\dx y
=n!\left|\frac{1}{n}K\right|=\frac{n!}{n^n}\cdot |K|.
\end{equation}
We claim that
\begin{align}\label{weaker-RBL}
n^n\int_{\R^n}f(y)^n\,\dx y\geq\prod_{i=1}^k\Big(\int_{E_i}f(x_i)\,\dx x_i\Big)^{c_i
}.
\end{align}
Equating the traces on the two sides of 
\eqref{BL-data-def}, we deduce
that $n_i=\dim E_i$ satisfy
\begin{equation}
\label{sumdin}
\sum_{i=1}^k\frac{c_in_i}{n}=1.
\end{equation}
For $z=\sum_{i=1}^kc_ix_i$ with $x_i\in E_i$, the log-concavity of $f$ and its definition \eqref{defoff} imply
\begin{equation}
\label{znsupf}
f\left (\frac z n\right )
=f\left (\sum_{i=1}^k \frac{c_in_i}{n}\cdot \frac{x_i}{n_i}\right )
\geq \prod_{i=1}^kf\left (\frac {x_i}{n_i}\right )^{\frac{c_in_i}{n}}= \prod_{i=1}^kf(x_i)^{\frac{c_i}{n}}.
\end{equation}
Now, the monotonicity of the integral and Barthe's inequality \eqref{RBL} with $f_i(x)=f(x)$ for $x\in E_i$ and $i=1,\ldots,k$ yield that
\begin{equation}
\label{sumPEsigmai0}
\int_{\R^n}f\left (\frac z n\right )^{n}\,\dx z\geq
\int_{\R^n}^*\sup_{z=\sum_{i=1}^kc_ix_i,\, x_i\in E_i}\prod_{i=1}^kf_i(x_i)^{c_i}\,\dx z
\geq\prod_{i=1}^k\Big(\int_{E_i}f_i(x_i)\,dx_i\Big)^{c_i}.
\end{equation}
Making the change of variable $y=\frac zn$, we conclude  \eqref{weaker-RBL}. Computing the right-hand side of \eqref{weaker-RBL},  we have
\begin{align}\label{compright}
\int_{E_i}f_i(x_i)\,\dx x_i=\int_{E_i}e^{-\|x_i\|_K}\,\dx x_i=\int_{E_i}e^{-\|x_i\|_{K\cap E_i}}\,\dx x_i=n_i!|K\cap E_i|
\end{align}
by \eqref{VKexponential} for $i=1,\ldots,k$. Therefore, \eqref{ftonint}, \eqref{weaker-RBL} and \eqref{compright} yield \eqref{Liakopoulos-RBL-ineq}.

Now, assume that equality holds in \eqref{Liakopoulos-RBL-ineq},
and hence equality holds in the second inequality of \eqref{sumPEsigmai0}, where we applied Barthe's inequality.
Since each $f_i$ is log-concave, $i=1,\ldots,k$,  and $\supp f_i=E_i$,
it follows from Theorem~\ref{RBLtheoequa} and Lemma~\ref{RBLtheoequa-remarks} that 
\begin{equation}
\label{RBLtheoequaform-Liakopoulos0}
f_i(x)=\theta_i e^{-\langle AP_{F_{\dep}}x,P_{F_{\dep}}x\rangle}\prod_{F_j\subset E_i}h_{j}\left(P_{F_j}(x)\right) 
\end{equation}
for $x\in E_i$, where
\begin{itemize}
\item[(i)] $\theta_i>0$,
\item[(ii)] $h_{j}\in L_1(F_j)$ is  log-concave for $j=1,\ldots,\ell$, and hence bounded,
\item[(iii)] $A \colon F_{\dep}\to F_{\dep}$ is a positive definite matrix such that
the eigenspaces of $A$ are critical subspaces.
\end{itemize}
Since each  $f_i$ has its unique maximum at $o$ which equals $1$, we may assume that each $h_j$ has its unique maximum at  $o$ which is also equal to $1$. It follows that $h_j=e^{-\varphi_j}$ for an even convex function $\varphi_j:F_j\to[0,\infty)$ with $\varphi_j(o)=0$, $j=1,\ldots,\ell$, and hence
for any $i=1,\ldots,k$, \eqref{RBLtheoequaform-Liakopoulos0} reads as
$$
f_i(x)= e^{-\langle AP_{F_{\dep}}x,P_{F_{\dep}}x\rangle}\prod_{F_j\subset E_i}\exp\left(-\varphi_j\left(P_{F_j}(x)\right)\right),
$$
and  any $x\in E_i$ satisfies that
\begin{equation}
\label{RBLtheoequaform-Liakopoulos}
\|x\|_K=\langle AP_{F_{\dep}}x,P_{F_{\dep}}x\rangle+\sum_{F_j\subset E_i}\varphi_{j}\left(P_{F_j}(x)\right).
\end{equation}

Next, we claim that
\begin{equation}
\label{Liakopoulos-equ-no-dep}
F_{\dep}=\{o\}.
\end{equation}
We suppose that there exists $\dim F_{\dep}\geq 1$, and seek a contradiction. As any $E_i$ can be written as
$$
E_i=\left(\oplus_{F_j\subset E_i}F_j\right)\oplus (F_{\dep}\cap E_i),
$$
it follows from \eqref{highdimcond-Liakopoulos} that
\begin{equation}
\sum_{i=1}^kc_iP_{E_i\cap F_{\dep}}=\left.I_n\right|_{F_{\dep}}.
\end{equation}
Therefore, there exists some $E_i$ such that $E_i\cap  F_{\dep}\neq \{o\}$, and let $x_0\in E_i\cap  F_{\dep}\neq \{o\}$. We deduce from \eqref{RBLtheoequaform-Liakopoulos} and $P_{F_j}(x_0)=o$ if $F_j\subset E_i$ that
$$
0=\lim_{t\to\infty}\frac{\|tx_0\|_K}{t^2}=\lim_{t\to\infty}\frac{\langle A(tx_0),(tx_0)\rangle}{t^2}=\langle Ax_0,x_0\rangle>0,
$$
and this contradiction proves \eqref{Liakopoulos-equ-no-dep}. 

 In turn, we deduce from \eqref{Liakopoulos-equ-no-dep}, the definition of $F_{\rm dep}$ and \eqref{RBLtheoequaform-Liakopoulos} that
\begin{align}
\label{RBLtheoequaform-Liakopoulos-i}
\oplus_{j=1}^\ell F_j=&\R^n,\\
\label{RBLtheoequaform-Liakopoulos-ii}
\|x\|_K=&\sum_{F_j\subset E_i}\varphi_{j}\left(P_{F_j}(x)\right)\mbox{ for any $x\in E_i$.}
\end{align}
Now, \eqref{RBLtheoequaform-Liakopoulos-i} proves part (i) of Theorem~\ref{Liakopoulos-RBL-equa}.

In order to verify (ii), notice that if $x\in F_m$ for an $m\in\{1,\ldots,\ell\}$, then $\varphi_{j}\left(P_{F_j}(x)\right)=\varphi_j(o)=0$ for $j\neq m$, and hence \eqref{RBLtheoequaform-Liakopoulos-ii} yields that $\varphi_{m}(x)=\|x\|_K$. We deduce again from \eqref{RBLtheoequaform-Liakopoulos-ii} that if $x\in E_i$, then
\begin{equation}
\label{RBLtheoequaform-Liakopoulos-Ei}
\|x\|_K=\sum_{F_j\subset E_i}\left\|P_{F_j}(x)\right\|_K=\sum_{j=1}^\ell\left\|P_{F_j}(x)\right\|_K,
\end{equation}
where we used   that either $F_j\subset E_i$ or $E_i\subset F_j^\bot$ for any $j\in\{1,\ldots,\ell\}$.

Let
\begin{align*}
M=\conv \{K\cap F_j\}_{j=1,\ldots\ell}.
\end{align*}
It follows from \eqref{sumMjnorm} that if $z\in\R^n$, then
\begin{equation}
\label{RBLtheoequaform-Liakopoulos-zM}
\|z\|_M=\sum_{j=1}^\ell\left\|P_{F_j}(z)\right\|_K.
\end{equation}
Since $M\subset K$, we deduce  that
\begin{equation}
\label{RBLtheoequaform-Liakopoulos-xKless-xM}
\|z\|_K\leq \|z\|_M
\end{equation}
 for any $z\in\R^n$. 
In order to prove $\|z\|_K\geq \|z\|_M$, we note that equality in \eqref{Liakopoulos-RBL-ineq} yields equality in the first inequality in \eqref{sumPEsigmai0} by the argument above.  It follows that if $z\in\R^n$, then
$$
\left(e^{-\|\frac zn\|_K}\right)^n=\sup_{z=\sum_{i=1}^kc_ix_i,\, x_i\in E_i}\prod_{i=1}^ke^{-c_i\|x_i\|_K};
$$  
or, equivalently,
\begin{equation}
\label{znormKsEi}
\|z\|_K= \inf_{z=\sum_{i=1}^k c_ix_i,\, x_i\in E_i}\sum_{i=1}^k c_i\|x_i\|_K
=\inf_{z=\sum_{i=1}^k y_i,\, y_i\in E_i}\sum_{i=1}^k \|y_i\|_K.
\end{equation}
Since we may assume that $\sum_{i=1}^k \|y_i\|_K<\|z\|_K+1$ in \eqref{znormKsEi}, and hence $ \|y_i\|_K<\|z\|_K+1$ for $i=1,\ldots,k$, we deduce, by compactness, the existence of
 $y_i\in E_i$ for $i=1,\ldots,k$ such that
\begin{equation}
\label{zxicond}
z=\sum_{i=1}^k y_i\mbox{ \ and \ }\|z\|_K=\sum_{i=1}^k \|y_i\|_K.
\end{equation} 
Now, apply \eqref{zxicond}, \eqref{RBLtheoequaform-Liakopoulos-Ei}, the triangle inequality, then \eqref{zxicond} again, and finally \eqref{RBLtheoequaform-Liakopoulos-zM} to obtain that
\begin{align*}
\|z\|_K=&\sum_{i=1}^k \|y_i\|_K
=\sum_{i=1}^k \sum_{j=1}^\ell\left\|P_{F_j}(y_i)\right\|_K=
\sum_{j=1}^\ell \sum_{i=1}^k\left\|P_{F_j}(y_i)\right\|_K\\
\geq&\sum_{j=1}^\ell \left\|\sum_{i=1}^k P_{F_j}(y_i)\right\|_K =\sum_{j=1}^\ell \left\|P_{F_j}(z)\right\|_K=\left\|z \right\|_M.
\end{align*} 
Therefore, $\|z\|_K= \|z\|_M$ for any $z\in \R^n$ by \eqref{RBLtheoequaform-Liakopoulos-xKless-xM}, which proves part (ii), thus completing the proof of Theorem~\ref{Liakopoulos-RBL-equa}.

\subsection*{Acknowledgments} We are grateful to the two reviewers whose remarks substantially improved the presentation.

K\'aroly J. B\"or\"oczky was partially supported by Hungarian NKFIH grant no.~150613.
The research of Ferenc Fodor was supported by NKFIH project no.~150151, which has been implemented with the support provided by the Ministry of Culture and Innovation of Hungary from the National Research, Development and Innovation Fund, financed under the ADVANCED\_24 funding scheme.


\end{document}